\documentclass{compositio}
\usepackage{amssymb,dsfont,bbold}
\usepackage{amsmath,amsthm,amssymb, mathrsfs,enumerate,wasysym}
 \usepackage{xypic}

\usepackage[all]{xy}
\usepackage[OT2,T1]{fontenc}
\DeclareSymbolFont{cyrletters}{OT2}{wncyr}{m}{n}
\DeclareMathSymbol{\Sha}{\mathalpha}{cyrletters}{"58}

\renewcommand{\contentsname}{Contents\\{\footnotesize\normalfont(A table
of contents should normally not be included)}}

\newcommand{\Q}{\mathbb{Q}}

\newcommand{\Z}{\mathbb{Z}}
\newcommand{\N}{\mathbb{N}}

\newcommand{\links}{\left(\begin{array}{cc}}
\newcommand{\rechts}{\end{array}\right)}
\newcommand{\bai}{\left[\begin{array}{cc}}
\newcommand{\dai}{\end{array}\right]}
\newcommand{\hidari}{\left(\begin{array}{c}}
\newcommand{\migi}{\end{array}\right)}

\newcommand{\CCC}{{\mathcal C}}

\newcommand{\Xloc}{\X_{loc}}

\newcommand{\cha}{{\mathrm{Char}}}
\newcommand{\Tam}{{\mathrm{Tam}}}
\newcommand{\sha}{{\mathrm{sha}}}
\newcommand{\Gal}{{\mathrm{Gal}}}

\newcommand{\Hom}{{\mathrm{Hom}}}

\newcommand{\coker}{{\mathrm{coker}\,}}

\newcommand{\GL}{{\mathrm{GL}}}
\newcommand{\ord}{{\mathrm{ord}}}
\newcommand{\col}{{\mathrm{Col}}}

\newcommand{\im}{{\mathrm{Im}\,}}
\newcommand{\Sel}{{\mathrm{Sel}}}
\newcommand{\Tr}{{\mathrm{Tr}}}

\newcommand{\gp}{{\mathfrak p}}

\newcommand{\gm}{{\mathfrak m}}

\renewcommand{\H}{{\mathcal{H}}}
\renewcommand{\qedsymbol}{{$\mathit{QED}$}}

\def\Iw{\mathop{\mathrm{Iw}}\nolimits}
\renewcommand{\phi}{{\varphi}}

\newcommand{\G}{{\mathcal G}}
\newcommand{\X}{{\mathcal X}}

\newtheorem{theorem}{Theorem}[section]

\newtheorem{algorithm}[theorem]{Algorithm}

\newtheorem{conjecture}[theorem]{Conjecture}

\newtheorem{corollary}[theorem]{Corollary}

\newtheorem{lemma}[theorem]{Lemma}
\newtheorem{propn}[theorem]{Proposition}
\newtheorem{proposition}[theorem]{Proposition}
\newtheorem{open problem}[theorem]{Open Problem}

\theoremstyle{definition}

\newtheorem{definition}[theorem]{Definition}

\newtheorem{remark}[theorem]{Remark}

\newtheorem{maintheorem}[theorem]{Main Theorem}

\usepackage{fancyhdr}
\begin{document}

\title{The \v{S}afarevi\v{c}-Tate group in cyclotomic $\Z_p$-extensions at supersingular primes}
\author{Florian Sprung (``ian")}
\email{ian.sprung@gmail.com}
\address{151 Thayer Street, Box \#1917, Providence, RI 02912, USA}
%

%
%\dedication{A dedication can be included here.}
\classification{11G05 (primary), 11R23 (secondary)}
\keywords{elliptic curve, supersingular reduction, Selmer group, Tate-Shafarevich group, Iwasawa Theory}
%\thanks{, and we thank whoever introduced the Topics Exam at Brown University.}

\begin{abstract}
We study the asymptotic growth of the $p$-primary component of the \v{S}afarevi\v{c}-Tate group in the cyclotomic direction at any odd prime of good supersingular reduction, generalizing work of Kobayashi. This explains formulas obtained by Kurihara, Perrin-Riou, and Nasybullin in terms of Iwasawa invariants of modified Selmer groups. \end{abstract}

\maketitle

%\vspace*{6pt}\tableofcontents  % for this guide only.
% A table of contents should normally not be included

\begin{section}{Introduction}
While the \v{S}afarevi\v{c}-Tate group came about when shedding the light of cohomology on the arithmetic of elliptic curves, it has also created new mysteries. One of them is its size, which we know to be finite over $\Q$ only when the analytic rank is $0$ or $1$, although this hasn't stopped mathematicians from incorporating it into the conjecture of Birch and Swinnerton-Dyer (BSD). Another mystery is how the size should behave as one varies the number field at hand. In the 1970's, Mazur initiated an answer to this question in the case of $\Z_p$-extensions when $p$ has good ordinary reduction - comparing favorably to versions of BSD along these $\Z_p$-extensions. 

The supersingular case (where $p$ divides $a_p=p+1-\#E(\mathbb{F}_p)$) has different growth patterns and was treated by Kurihara \cite{kurihara} under certain assumptions in the case of cyclotomic $\Z_p$-extensions, but it was Perrin-Riou \cite{perrinriou} who discovered a general growth pattern for cyclotomic $\Z_p$-extensions involving two pairs of (unspecified) constants. Natural questions to ask are: Why does this growth pattern occur? Where do her constants come from? Kobayashi \cite{kobayashi} has explained each pair of constants \textit{under the assumption $a_p=0$} as algebraic Iwasawa invariants of his two modified Selmer groups. His work compared favorably to the formulas of Pollack \cite{pollack} (who also assumed $a_p=0$) of special values of $L$-functions involving analytic Iwasawa invariants. 

The general supersingular case, however, has seemed less amenable to analysis so far. In this paper, we give \textit{three} possible growth formulas (which illlustrates the difficulty for the case $a_p\neq0$) involving algebraic Iwasawa invariants coming from two modified Selmer groups introduced in earlier work \cite{shuron} (think of them as Selmer groups coming from ``half" the local points): In one of the formulas, Perrin-Riou's two pairs of constants do indeed come from the two pairs of Iwasawa invariants (which happens for example when $a_p=0$, where we obtain Kobayashi's formula). The other two formulas involve just one pair of Iwasawa invariants - the \v{S}afarevi\v{c}-Tate group is \textit{modest} and chooses whichever pair makes it grow more slowly. 
 
 We also generalize Perrin-Riou's work to give the growth pattern of the \v{S}afarevi\v{c}-Tate group in the tower of number fields obtained by adjoining all $p$-power roots of unity to $\Q$. A growth formula for this tower had been announced by Nasybullin \cite{nasybullin} in the 1970's, but no proofs were given. The results in this paper compare favorably with formulas on the analytic side, which appear in a twin paper \cite{sprung} to this one. %given by Tate, and also Manin (CHECK HISTORY) 
 
More concretely, let $E$ be an elliptic curve over $\Q$ and $p$ an odd prime of good reduction, and let $\zeta_{p^{n+1}}$ be a primitive $p^{n+1}$-th root of unity. We consider the cyclotomic $\Z_p$-extension of $\Q$, i.e. the tower $\Q=\Q_0\subset \Q_1\subset\cdots \subset \Q_n\subset \cdots \subset \Q_\infty=\bigcup_n \Q_n$, where $\Q_n$ is the subfield of $\Q(\zeta_{p^{n+1}})$ so that $\Gal(\Q_n/\Q)\cong \Z/p^n\Z$. Denote the Galois group of the entire cyclotomic extension by $\Gamma:=\Gal(\Q_\infty/\Q)\cong\Z_p$ and define the \v{S}afarevi\v{c}-Tate group at level $n$:
\[\Sha(E/\Q_n):=\ker\left(H^1(\Q_n,E)\rightarrow \prod_v H^1(\Q_{n,v},E)\right),\]
where $v$ ranges over all places of $\Q_n$. Assume for the rest of the paper that the $p$-primary component of $\Sha(E/\Q_n)$ is finite, and denote its size by $p^{e_n}:=\#\Sha(E/\Q_n)[p^\infty]$.

When $p$ is ordinary, i.e. $p \nmid a_p=p+1-\#E(\mathbb{F}_p)$, Mazur discovered that
\[e_n-e_{n-1}=(p^n-p^{n-1})\mu+\lambda -r_\infty \text{ for $n \gg 0$},\]
which is an analogue of a classical result by Iwasawa on the $p$-part of the class number of $\Q_n$. Here, $r_\infty$ is the rank of $E(\Q_\infty)$ (which is finite by theorems of Kato \cite{kato} and Rohrlich \cite{rohrlich} even when $p$ is supersingular), and $\mu$ and $\lambda$ are the Iwasawa invariants of the Pontryagin dual of the $p$-primary Selmer group $\Sel_p(E/\Q_\infty)$ over $\Q_\infty$. This is the following kernel:
\[\Sel_p(E/\Q_\infty):=\ker\left(H^1(\Q_\infty,E[p^\infty])\rightarrow \prod_v H^1(\Q_{\infty,v},E[p^\infty])/E(\Q_{\infty,v})\otimes\Q_p/\Z_p\right).\] This relates to (the residue part of) the Birch and Swinnerton-Dyer conjecture for $E(\Q_n)$, denoted $BSD_{p^n}$, as follows: This $BSD_{p^n}$ states that $\#\Sha(E/\Q_n)$ should equal
\[\sha_n:=\lim_{s\rightarrow 1}\frac{L(E/\Q_n,s)}{(s-1)^{r'_n}}\frac{\#E^{tors}(\Q_n)^2\sqrt{D(\Q_n)}}{\Omega_{E/\Q_n}R(E/\Q_n)\Tam(E/\Q_n)},\]
where $D(\Q_n)$ is the discriminant, $R(E/\Q_n)$ the regulator, $\Tam(E/\Q_n)$ the product of the Tamagawa numbers, $\Omega_{E/\Q_n}$ the real period and $r_n'$ the order of vanishing of $L(E/\Q_n,s)$, conjecturally equal to the rank of $E(\Q_n)$. Now we know by Rohrlich \cite{rohrlich} that $r_\infty'=\lim_{n\rightarrow \infty} r_n'$ exists. When $p$ is ordinary, one can then show that
\[\ord_p(\sha_n)-\ord_p(\sha_{n-1})=(p^n-p^{n-1})\mu_{an}+\lambda_{an}-r_\infty' \text{ for $n \gg 0$},\]
where the integers $\mu_{an}$ and $\lambda_{an}$ are Iwasawa invariants of the $p$-adic $L$-function associated to the elliptic curve. The Iwasawa Main Conjecture would imply that $\mu_{an}=\mu$ and $\lambda_{an}=\lambda$, thereby proving the $p$-part of $BSD_{p^n}$ for all $n\geq N$ provided the $p$-part of $BSD_{p^N}$ is true for some large $N$.  A proof of the Iwasawa Main Conjecture has been given by Skinner and Urban \cite{skinnerurban}.

The main result of this paper is the following. 
\begin{maintheorem} Let $p$ be an odd supersingular prime, i.e. $p|a_p$. Assuming that $\Sha(E/\Q_n)[p^\infty]$ is finite, let $p^{e_n}:=\#\Sha(E/\Q_n)[p^\infty]$, and let $r_\infty$ be the rank of $E(\Q_\infty)$. Further, put
 \[q_n^\sharp:=\begin{cases} p^{n-1}-p^{n-2}+p^{n-3}-p^{n-4}+\cdots +p^2-p \text{ if $n$ is odd},\\p^n-p^{n-1}+p^{n-2}-p^{n-3}+\cdots +p^2-p \text{ if $n$ is even.}\end{cases} \]
\[q_n^\flat:=\begin{cases}  p^n-p^{n-1}+p^{n-2}-p^{n-3}+\cdots -p^2+p-1\text{ if $n$ is odd},\\p^{n-1}-p^{n-2}+p^{n-3}-p^{n-4}+\cdots +p-1 \text{ if $n$ is even.}\end{cases}\]
 Then there are integer invariants $\mu_\sharp,\lambda_\sharp,\mu_\flat,\lambda_\flat$ so that for $n\gg 0$, 
\[e_n-e_{n-1}=\begin{cases}(p^n-p^{n-1})\mu_\sharp+\lambda_\sharp -r_\infty+q_n^\sharp \text{ if $n$ is even}\\ (p^n-p^{n-1})\mu_\flat+\lambda_\flat -r_\infty+q_n^\flat \text{ if $n$ is odd}\end{cases}\]
when $a_p=0$ or $\mu_\sharp=\mu_\flat$, and
\[e_n-e_{n-1}=(p^n-p^{n-1})\mu_*+\lambda_* -r_\infty+q_n^*\]
 when $a_p\neq 0$ and $\mu_\sharp\neq \mu_\flat$, where $*\in \{\sharp,\flat\}$ is chosen so that $\mu_*=\min(\mu_\sharp,\mu_\flat)$. 
\end{maintheorem}
These formulas are very similar to Perrin-Riou's, but our methods allow us to interpret $\mu_\sharp,\lambda_\sharp$, and $\mu_\flat,\lambda_\flat$ as the $\mu$- and $\lambda$-invariants of the  Pontryagin duals $\X^\sharp$ resp. $\X^\flat$ of the modified Selmer groups $\Sel^\sharp$ resp. $\Sel^\flat$ defined in \cite{shuron}. Let $\Q_{\infty,p}$ (resp. $\Q_{n,p}$) be $\Q_\infty$ (resp. $\Q_{n}$) completed at (the prime above) $p$. The reason why we don't work with the Selmer group $\Sel(E/\Q_\infty)$ is that the local condition coming from all $\Q_{\infty,p}$-rational points of $E$ is too weak in the supersingular case, making its Pontryagin dual not $\Z_p[[\Gamma]]$-torsion. The two modified Selmer groups each have a stronger local condition $E_\infty^{\sharp/\flat}$ at $p$, which intuitively can be thought of as half-$\Q_{\infty,p}$-rational points of $E$:
\[\Sel_p^\sharp(E/\Q_\infty):=\ker\left(H^1(\Q_\infty,E[p^\infty])\rightarrow H^1(\Q_{\infty,p},E[p^\infty])/E_{\infty,\gp}^\sharp\right).\] 
$E_{\infty,\gp}^\sharp$ is the exact annihilator with respect to the Tate pairing of the kernel of a map $\col^\sharp:H^1(\Q_{\infty,p},T)\rightarrow \Z_p[[\Gamma]]$. This $\col^\sharp$ and its companion $\col^\flat$ each encode half the information of the local points as they \textit{together} know the geometry of $E(\Q_{\infty,p})\otimes \Q_p/\Z_p$: They respect the interplay of pairs of generators (constructed by Kobayashi) for the local points $E(\Q_{n,p})$ as Galois modules for varying $n$. When $\X^{\sharp/\flat}$ is not $\Z_p[[\Gamma]]$-torsion, we put $\mu^{\sharp/\flat}:=\infty$. Conjecturally, this never happens and we know it happens for at most one of the $\sharp,\flat$ (\cite[Theorem 7.14]{shuron}). %$\lambda_\sharp$ and $\lambda_\flat$ are the corresponding $\lambda$-invariants. 
Kobayashi has proved this theorem for the case $a_p=0$ in \cite{kobayashi}, where $\sharp=-,\flat=+$. %We should also remark that all this is a generalization of earlier work by Kurihara (\cite{kurihara}).

When $a_p\neq0$, the \v{S}afarevi\v{c}-Tate group grows using the \textit{smaller} one of the two invariants $\mu_{\sharp/\flat}$ if they are indeed different. We call this phenomenon \textit{modesty}.

From BSD considerations, one should have an analogous formula on the analytic side, which is a result of a twin paper to this one (concerning modular forms of weight two). When $p$ is supersingular, there are two integral $p$-adic $L$-functions with Iwasawa invariants $\mu^\sharp_{an}, \mu^\flat_{an},\lambda^\sharp_{an},\lambda^\flat_{an}$ (see \cite{pollack}\footnote{In Pollack's notation, we have $\sharp=+,\flat=-$. Note the sign switch with Kobayashi's notation!} for $a_p=0$ and the twin paper \cite{sprung} for any $p|a_p$). Further, we have for $n \gg 0$
\[\ord_p(\sha_n)-\ord_p(\sha_{n-1})=\begin{cases}(p^n-p^{n-1})\mu^\sharp_{an}+\lambda^\sharp_{an} -r_\infty+q_n^\sharp \text{ if $n$ is even}\\ (p^n-p^{n-1})\mu^\flat_{an}+\lambda^\flat_{an} -r_\infty+q_n^\flat \text{ if $n$ is odd}\end{cases}\]
when $a_p=0$ or $\mu^\sharp_{an}=\mu^\flat_{an}$, and
\[\ord_p(\sha_n)-\ord_p(\sha_{n-1})=(p^n-p^{n-1})\mu^*_{an}+\lambda^*_{an} -r_\infty+q_n^*,\]
where $*\in \{\sharp,\flat\}$ is chosen so that $\mu^*_{an}=\min(\mu^\sharp_{an},\mu^\flat_{an})$ when $a_p\neq 0$ and $\mu_\sharp^{an}\neq \mu_\flat^{an}$. 

 Pollack proved this formula when $a_p=0$ in \cite{pollack}. For general $p|a_p$, see \cite{sprung}. When the $p$-adic representation $\Gal(\overline{\Q}/\Q) \rightarrow \GL_{\Z_p}(T)$ on the automorphism group of the $p$-adic Tate module $T$ is surjective, we know half the Iwasawa Main Conjecture (see \cite[Theorem 1.4]{shuron} for a precise statement). This means that half of the $p$-part of $BSD_{p^N}$ for some large $N$ implies half of the $p$-part of $BSD_{p^n}$ for all $n \geq N$. 

%Mazur's method was to control Selmer groups along the $\Z_p$-extension to prove his theorem in the ordinary case. This method made it possible to define the Iwasawa invariants $\mu$ and $\lambda$, but breaks down when $p$ is supersingular. Although Kobayashi invented a control theore analogous method to control his $\pm$ Selmer groups in the case $a_p=0$, he didn't make use of it to prove 
 One of Mazur's tools for proving his formula for $e_n-e_{n-1}$ in the ordinary case was a theorem (gControl Theorem") concerning the Galois theoretic behavior of Selmer groups. While Kobayashi proved a supersingular analogue of this, he needed a different method to prove his formula for $e_n-e_{n-1}$: He decomposed $E(\Q_{n,p})$ into two parts $E^+(\Q_{n,p})$ and $E^-(\Q_{n,p})$ (with constant small overlap), and used this decomposition to prove nice properties of the images of two maps $\col_n^+$ and $\col_n^-$ (``Coleman maps") to obtain the two terms responsible for $e_n-e_{n-1}$. In this paper, we present a more genuine ``Coleman map method" to be able to treat the general supersingular case. The idea is to expand and generalize Kobayashi's treatment so well that we arrive at the nice properties without having to decompoe $E(\Q_{n,p})$ into two parts (which seems difficult for $a_p\neq0$). %This is Proposition \ref{auxiliary}.%, we scrutinize a map $\col_n$ (which coincides with the vector of maps $(\col_n^+,\col_n^-)$ defined in \cite{kobayashi} if $a_p=0$) to obtain our estimates.

\bf{Organization of Paper}. \rm In Section \ref{thecolemanmaps}, we review the Iwasawa theory for elliptic curves at supersingular primes as presented in \cite{shuron}. In particular, we explain how the Coleman maps $\col_n$ come about and give rise to the two modified Selmer groups $\Sel^\sharp$ and $\Sel^\flat$. In Section \ref{themainideas}, we explain the main ideas on how to derive our result from the theory presented in Section \ref{thecolemanmaps}, assuming the most important proposition (the Modesty Proposition \ref{modestypropn}) in this paper. It is this proposition that explains why the \v{S}afarevi\v{c}-Tate Group is so modest (i.e. chooses the smaller $\mu$-invariant) when growing in our tower of number fields. Its proof constitutes Section \ref{proofofthemodestyproposition}, which forms the technical heart of the paper and contains the most important new ideas. One idea is that of \textit{valuation matrices}, introduced in Definition \ref {valuationmatrix}, which allows us to estimate $p$-adic valuations of special values of appropriate functions. The other is a description of the cokernel of the map $\col_n$, obtained in Proposition \ref{auxiliary}, that ultimately allows us to use the estimates obtained from valuation matrices to describe the growth of the \v{S}afarevi\v{c}-Tate Group. In the last section, we explain the formulas of Perrin-Riou in terms of our Iwasawa invariants in the case of her setting.

A related question to this paper is the case of anticyclotomic extensions. \c{C}iperiani has proved in \cite{ciperiani} that the $p$-primary part of the \v{S}afarevi\v{c}-Tate in the full anticyclotomic extension has corank zero as an Iwasawa module under certain assumptions that in particular force $a_p=0$.

Another question is how to generalize to modular forms of higher weight: Lei, Loeffler, and Zerbes \cite{llz} have constructed Coleman maps in this context, using the theory of Wach modules, and working with $\Z_p[[\Gamma]]\otimes \Q$. They construct Selmer groups that depend on the basis of their Wach module (which in particular may make their Iwasawa invariants depend on the chosen basis as well). An explicit version of their work, hopefully independent of any bases, should thus lead to nice generalizations of the results of this paper.

%\textit{Acknowledgments} [Thank people here]
\end{section}
\begin{section}{The Coleman Maps}\label{thecolemanmaps}
We first fix some notation. Let $n\in \N$ be an integer, %$[\frac{a}{b}]$ is the greatest integer not greater than $\frac{a}{b}$[DID THIS EVER END UP COMING UP ??], 
$I$ the two-by-two identity matrix, and ${\Phi_n(X)=\displaystyle \sum_{t=0}^{p-1} X^{p^{n-1}t}}$ the $p^n$-th cyclotomic polynomial, which is the irreducible polynomial for any primitive $p^n$-th root of unity $\zeta_{p^n}$. %\in \MU_{p^n}$.
Put $k_n=\Q_p(\zeta_{p^{n+1}})$, and let $\G_n=\Gal(\Q(\zeta_{p^{n+1}})/\Q)\cong\Gal(k_n/\Q)$. We then have the decomposition $\G_n\cong (\Z/{p^{n+1}\Z})^\times \cong\Delta \times \Gamma_n$, where $\Delta\cong \Z/(p-1)\Z$ since $p$ is odd, and $\Gamma_n\cong \Z/p^n\Z$. 
 Taking inverse limits, put $\G_\infty= \displaystyle\varprojlim_n \G_n \cong \Delta \times \Gamma$, where $\Gamma\cong \Z_p$. Now fix a topological generator $\gamma$ of $\Gamma$. By sending $\gamma$ to $(1+X)$, we can identify $\Lambda=\Z_p[[\G_\infty]]$ with $\Z_p[\Delta][[X]]$.
 Denoting the image of $\gamma$ under the projection $\G_\infty \rightarrow \G_n$ by $\gamma_n$, we can similarly identify $\Lambda_n=\Z_p[\G_n]$ with $\Z_p[\Delta][[X]]/(\omega_n(X))$, where $\omega_n(X)= (1+X)^{p^n}-1$. See \cite[Chapter Seven]{washington}.
 We write $T$ for the $p$-adic Tate module, and $V=T\otimes \Q_p$.
 We denote the tower of global fields by capital letters: $K_n:=\Q(\zeta_{p^{n+1}})$, and $K_\infty:=\bigcup_nK_n$.
 
 Let $\gm_n$ be the maximal ideal of $\Z_p[\zeta_{p^{n+1}}]$.
%From Honda's theory of formal groups, there is an isomorphism $\F\rightarrow \widehat{E}$ (cf. Corollary 8.5 in \cite{kobayashi}). We can thus r
Regard the (formal) group $\widehat{E}(\gm_n)$ as a subgroup of $H^1(k_n,T)$ by the Kummer map. The cup product then gives us a pairing

\[(\,,\,)_n:\widehat{E}(\gm_n) \times H^1(k_n,T)\rightarrow H^2(k_n,\Z_p(1))\cong \Z_p.\]

For $x_n\in\widehat{E}(\gm_n)$, we now define a Galois equivariant morphism \[P_{x_n}:H^1(k_n,T)\rightarrow \Z_p[\G_n]\text{ by }z\mapsto \sum_{\sigma \in \G_n} (x_n^\sigma,z)_n\sigma.\]

\begin{theorem}[(Kobayashi \cite{kobayashi}, Section 8.4)] There are two generators $\delta_n^1 $ and $\delta_n^0$ for $\widehat{E}(\gm_n)$ as a $\Z_p[\G_n]$-module.
\end{theorem}
Kobayashi named these elements $c_n$ and $c_{n-1}$ (\cite[Section 2]{shuron} explains our notation).

\begin{definition}\label{Hn}Put $A:=\links a_p & 1 \\ -1 & 0 \rechts,$ $  \CCC_i:=\links a_p & 1 \\ -\Phi_{i}(1+X) & 0 \rechts,$ $\H_n:=-\CCC_1\cdots \CCC_nA^{-1}$ for $n>0$ and $\H_0:=-A^{-1}$. We denote right-multiplication with the matrix $\H_n$ by $h_n$.% : \[\Lambda_n\oplus \Lambda_n \xrightarrow{h_n} (\Lambda_n\oplus \Lambda_n)\H_n\subset \Lambda_n\oplus \Lambda_n.\]
\end{definition}

\theorem\label{colemanmap} \cite[Proposition 5.3]{shuron} There is a well-defined map $\col_n:H^1(k_n,T)\rightarrow L_n:= \frac{\Lambda_n \oplus \Lambda_n}{\ker h_n}$ so that
\[\xymatrix {H^1(k_n,T)\ar@/^8mm/@{->}[0,2]^{\hspace{2.5mm}P_{\delta_n^1}, P_{\delta_n^0}}\ar[r]^{\quad\col_n}& L_n \ar@{^(->}[r]^{h_n} &{\Lambda_n \oplus \Lambda_n}
}\text{ commutes.}
\] 
\theorem\label{compatibility}\cite[Corollary 5.6, Limit Proposition 5.7, Definition 7.1]{shuron}  The maps $(\col_n)_n$ form an inverse system with respect to the corestriction maps $H^1(k_{n+1},T)\rightarrow H^1(k_n,T)$ and are compatible with the projection maps $L_{n+1}\rightarrow L_n$. In the inverse limit, we have $\varprojlim_n L_n \cong \Lambda \oplus \Lambda$, so there is a splitting $\varprojlim_n \col_n =: (\col_\sharp, \col_\flat)$.

\begin{definition} Let $\gp_n$ be the place in $K_n$ and $\gp$ the place in $K_{\infty}$ above $p$ and let $E^{\sharp}_{\infty,\gp}$ (resp. $E^{\flat}_{\infty,\gp})$ be the exact annihilator of $\ker \col^{\sharp}$ (resp. $\ker \col^{\flat})$ under the local Tate pairing \[ {\displaystyle\varprojlim_nH^1(K_{n,\gp_n},T)\times \varinjlim_n H^1(K_{n,\gp_n},V/T)\rightarrow \Q_p/\Z_p}.\] %Note that these two annihilators are both inside $E(K_{\infty,\gp})\otimes \Q_p/\Z_p$ (which is the exact annihilator of $E(K_{\infty,\gp})\otimes \Z_p$), because $\displaystyle \varprojlim_m E(K_{n,\gp_n})^{\oplus 2}/p^m \subset \ker \col_n$ by construction.
\end{definition}

\begin{definition}We define $\X^{*}(E/K_\infty)=\Hom(\Sel^{*}(E/K_\infty),\Q_p/\Z_p)$ for $*\in\{\sharp,\flat\}$, where 
\[\Sel^{\sharp}(E/K_\infty):=\ker\left(\Sel(E/K_\infty)\longrightarrow \frac{E(K_{\infty,\gp})\otimes\Q_p/\Z_p}{E^{\sharp}_{\infty,\gp}}\right), \text{and}\]
\[\Sel^{\flat}(E/K_\infty):=\ker\left(\Sel(E/K_\infty)\longrightarrow \frac{E(K_{\infty,\gp})\otimes\Q_p/\Z_p}{E^{\flat}_{\infty,\gp}}\right).\]
\end{definition}
\begin{theorem} (\cite{kobayashi} for $a_p=0$, \cite{shuron} for general $a_p$) $\X^{*}(E/K_\infty)$ is $\Lambda$-cotorsion for at least one of $*\in\{\sharp,\flat\}$ when $a_p \neq 0$. When $a_p=0$, they both are.
\end{theorem}
\begin{conjecture} (\cite[Conjecture 7.15]{shuron}) $\X^{*}(E/K_\infty)$ are both $\Lambda$-cotorsion for $*\in\{\sharp,\flat\}$.
\end{conjecture}
\begin{definition} Let $\eta:\Delta\rightarrow \Z_p^\times$ be a character. For a $\Z_p[\Delta]$-module $M$, let $M^\eta$ be the $\eta$-component of $M$, i.e.
\[M^\eta=\varepsilon_\eta M,\text{ where } \varepsilon_\eta=\frac{1}{\#\Delta}\sum_{\tau \in \Delta}\eta^{-1}(\tau)\tau.\] 
\end{definition}
\begin{definition}
Let $*\in\{\sharp,\flat\}$. If $\X^*(E/K_\infty)^\eta$ is not finitely generated $\Z_p[[X]]$-torsion, we define $\mu_*^\eta:=\infty$. If $\X^*(E/K_\infty)^\eta$ is finitely generated $\Z_p[[X]]$-torsion, then we define $\mu_*^\eta$ as its Iwasawa $\mu$-invariant. In this case, we similarly define $\lambda_*^\eta$ as its Iwasawa $\lambda$-invariant.  
\end{definition}

\textup{It is these Iwasawa invariants that we will now use to estimate the growth of the \v{S}afarevi\v{c}-Tate groups in our towers of number fields.}

\remark In \cite{shuron}, $h_n$ is called $h_n^1$. Also, from the viewpoint of functional equations, it is more correct to look at completed versions of $\H_n$: $\widehat{\H_n}:=-\widehat{C_1}\cdots\widehat{C_n}A^{-1},$ where $\widehat{C_i}$ is $C_i$ with the lower left entry $-\widehat{\Phi_i(1+X)}=-\Phi_i(1+X)(1+X)^{\frac{-1}{2}p^{i-1}(p-1)}$. These matrices give rise to completed maps $\widehat{\col_n}, \widehat{\col_\sharp},$ and $\widehat{\col_\flat}$, and Selmer group duals $\widehat{\X_\sharp}$ and $\widehat{\X_\flat}$, and to a (completed) formulation of the supersingular Iwasawa main conjecture as in \cite{shuron}, but the $\mu-$ and $\lambda-$invariants of (the isotypical components of) $\widehat{\X_*}$  and ${\X_*}$ are the same for $*\in\{\sharp,\flat\}$, since $\widehat{\H_n}$ and $\H_n$ have the same \textit{valuation matrices} (defined in Definition \ref{valuationmatrix}). %This follows e.g. from \cite{sprung}(Lemma NUMBER), so we have decided to work with non-completed Coleman maps for simplicity.
\end{section}

\begin{section}{The Main Ideas}\label{themainideas}
\begin{subsection}{The Selmer group makes the \v{S}afarevi\v{c}-Tate group grow}
To get a handle on the \v{S}afarevi\v{c}-Tate group, we use the following well-known exact sequence:
\[0 \rightarrow E(K_n)\otimes \Q_p/\Z_p \rightarrow \Sel_p(E/K_n) \rightarrow \Sha (E/K_n)[p^\infty]\rightarrow 0\]
Here, $\Sel_p(E/K)$ and $\Sha(E/K)[p^\infty]$ denote the $p$-primary components of the Selmer group and the \v{S}afarevi\v{c}-Tate group over a number field $K$. We should thus be analyzing the Selmer group, for which we need the following convention:

\begin{definition}Let $M=(M_n)_{n=1,2,..}$ be an inverse system of finitely generated $\Z_p$-modules with transition maps $\pi_n:M_{n}\rightarrow M_{n-1}$. When $\pi_n$ has finite kernel and cokernel, we let the $n$-th \textit{Kobayashi rank} $  \nabla_n M$ be 
\[\nabla_n M :=\text{length}_{\Z_p}(\ker \pi_n)- \text{length}_{\Z_p}(\coker\pi_n)+\dim_{\Q_p}M_{n-1}\otimes \Q_p.\]
\end{definition}

\begin{definition} If $M$ is a finitely generated torsion $\Z_p[[X]]$-module, we define the $n$-th \textit{characteristic rank} as
\[\Upsilon_nM:=(p^n-p^{n-1})\ord_pf_M(\zeta_{p^n}-1),\]
where $f_M$ is a characteristic polynomial of $M$. 
\end{definition}

\remark\label{upsilonisnabla} Let  $M$ be a finitely generated torsion $\Z_p[[X]]$-module. Then for large $n$, the $n$-th characteristic rank $\Upsilon_nM$ equals the $n$-th Kobayashi rank of the inverse system $\left(M\otimes \frac{\Z_p[[X]]}{\omega_n(X)}\right)_n$ with the projection maps.
\proof This is Lemma 10.5 i in \cite{kobayashi}.

Now put \[p^{e_n^\eta}:= \# \Sha(E/K_n)^\eta[p^\infty].\]

Consider the inductive system $(\Sel(E/K_n)^\eta)_n$. Taking its Pontryagin dual, we obtain a projective system $\X^\eta:=(\X(E/K_n)^\eta)_n$. Now it is known that for sufficiently large $n$, the rank of the Mordell-Weil group $E(K_n)$ stabilizes to some $r_\infty^\eta=rk_{\Z_p}E(K_\infty)^\eta$ (e.g. \cite{kato}), so by the well-known exact sequence from above, we have 
\begin{equation}\label{enandselmer}e_n^\eta-e_{n-1}^\eta=\text{length}_{\Z_p}\left(\Sha(E/K_n)^\eta/\Sha(E/K_{n-1})^\eta\right)=\nabla_n\X^\eta-r_\infty^\eta.\end{equation}
\end{subsection}

\begin{subsection}{The $\sharp/\flat$ - Selmer groups make the Selmer group grow}
From now on, we will also need a global object.
\begin{definition}Recall that $K_n= \Q(\zeta_{p^{n+1}})$. We let  $K_{-1}=\Q$. Denoting by $j$ the natural morphism  Spec $K_n \rightarrow \text{Spec } O_{K_n}[\frac{1}{p}]$,  we define $\mathbf{H}^1(T):=\displaystyle \varprojlim_n$ $ H^q(\text{Spec }O_{K_n}[\frac{1}{p}],j_*T)$, where $H^q(\text{Spec }O_{K_n}[\frac{1}{p}],j_*T)$ is the $q$-th \'etale cohomology group.
\end{definition}
We would like to estimate $\nabla_n \X^\eta$, for which we consider the Kobayashi ranks of the following inverse systems:
\begin{definition}$\Xloc(E/K_n):=\coker(\mathbf{H}^1(T)_{\Gamma_n}\rightarrow H^1(k_n,T)/E(k_n)\otimes \Z_p),$ where the map is the quotient of the localization map. We consider the projective system
\[\Xloc:=\left(\Xloc(E/K_n)\right)_n\]
\end{definition}

\begin{definition}The \textit{fine Selmer group} is the following kernel:
\[\Sel_0(E/K_n):=\ker \left( H^1(K_n,E[p^\infty])\rightarrow \prod_v H^1(K_{n,v},E[p^\infty])\right)\]
We denote the inverse system $\varprojlim_n \Hom(\Sel_0(E/K_n), \Q_p/\Z_p)$ by $\X_0$. 

\end{definition}

\begin{propn}\label{kobayashisproposition} (Kobayashi \cite{kobayashi}) For large $n, \nabla_n \X^\eta = \nabla_n \Xloc^\eta + \nabla_n \X_0^\eta$.
\end{propn}
\begin{proof}This follows from the exact sequence \cite[(10.35)]{kobayashi} and the additivity of Kobayshi ranks (\cite[Lemma 10.4 i]{kobayashi}) in them. For his notation, note that $\Xloc(E/K_n)$ is $\mathcal{Y}'(E/K_n)$, and cf. \cite[Lemma 10.6]{kobayashi}.
\end{proof}

We won't compute either of the Kobayashi ranks, but decompose this sum into more, difficult summands, which together yield an easy sum.
\begin{definition}Let $[\frac{a}{b}]$ be the greatest integer not greater than $\frac{a}{b}$.
\[q_n^\sharp:=\left[\frac{p^n}{p+1}\right] = p^{n-1}-p^{n-2}+p^{n-3}-p^{n-4}+\cdots +p^2-p \text{ if $n$ is odd.}\]\[
q_n^\flat:=%p^n-p^{n-1}+p^{n-2}-p^{n-3}+\cdots -p^2+p-1\text{ if $n$ is odd},\\
\left[\frac{p^n}{p+1}\right]=p^{n-1}-p^{n-2}+p^{n-3}-p^{n-4}+\cdots +p-1 \text{ if $n$ is even.}\]
We also define $q_n^\sharp:=q_{n+1}^\sharp$ for even $n$, and $q_n^\flat:=q_{n+1}^\flat$ for odd $n$.
\end{definition}

%We denote by $\Lambda/\textrm{\col}_*$ the projective system $(\Lambda_n/im(\textrm{\textrm{\col}}_*))_n$ (OR $\textrm{\col}_n$?? CHECK) for $*=\sharp$ rec. $*=\flat$. ALSO, in what follows stuff is off by 1 for the trivial CHARACTER. put in $^\eta$!!!
\begin{algorithm}[(Modesty Algorithm)]\label{modestyalgorithm}
Let $n$ be a given integer. For a character $\eta:\Delta\to \Z_p^\times$, we choose an element $*\in\{\sharp,\flat\}$ as follows: Let $\mathbf{e}$ be a basis of the rank one $\Lambda$-module $\mathbf{H}^1(T)$. We regard $\mathbf{e}$ as an element of $\mathbf{H}^1_{\Iw}(T)=\varprojlim_n H^1(k_n,T)$ by the localization map. Denote by $\mu(\col_\sharp^\eta)$ the $\mu$-invariant of $\col_\sharp^\eta(\mathbf{e})$ if $\col_\sharp^\eta(\mathbf{e})\neq 0$, and we let $\mu(\col_\sharp^\eta):=\infty$ if $\col_\sharp^\eta(\mathbf{e})= 0$. We define $\mu(\col_\flat^\eta)$ analogously. Now:
\itemize
\item If $a_p=0$ or $\mu(\col_\sharp^\eta)=\mu(\col_\flat^\eta), $ then $*:=\begin{cases}\sharp\text{ if $n$ is odd,}\\ \flat \text{ if $n$ is even}. \end{cases}$

\item If $a_p\neq 0 $ and $\begin{cases}\mu(\col_\sharp^\eta)<\mu(\col_\flat^\eta)\text{, then }*:=\sharp,\\ \mu(\col_\sharp^\eta)>\mu(\col_\flat^\eta)\text{, then }*:=\flat.\end{cases}$
\end{algorithm}
\begin{propn}[(Modesty Proposition)]\label{modestypropn} Given a character $\eta:\Delta\to \Z_p^\times$ and $n\in\N$, pick $*$ according to the Modesty Algorithm \ref{modestyalgorithm}. Let $\mathbf{e}$ be as above. Then  for large $n$, \[\nabla_n \Xloc^\eta=q_n^* + \Upsilon_n (\Lambda/\col_*(\mathbf{e}))^\eta.\]
\end{propn}

\begin{proof}We will devote the next section to the proof.
\end{proof}

\begin{propn}\label{above} Given $\eta:\Delta\to \Z_p^\times$ and $n\in\N$,  pick $*$ according to the Modesty Algorithm \ref{modestyalgorithm} and let $\mathbf{e}$ be as above. Then \[\Upsilon_n \X_*^\eta = \begin{cases}\Upsilon_n (\Lambda/\col_*(\mathbf{e}))^\eta + \Upsilon_n \X_0^\eta &\text{ if } \eta\text{ is trivial or } a_p\neq0,\\ \Upsilon_n (\Lambda/\col_*(\mathbf{e}))^\eta + \Upsilon_n \X_0^\eta-1&\text{ if } \eta\text{ is nontrivial and } a_p=0.\end{cases}\]
\end{propn}
\begin{proof} Note that we have chosen $*$ so that $\X_*^\eta$ is a finitely generated torsion $\Z_p[[X]]$-module. From the exact sequences \cite[in Proposition 7.19]{shuron}, \cite[Proposition 10.6 ii]{kobayashi}, and the fact (\cite[Proposition 7.6]{shuron}) that $\im\col_\sharp^\eta=J^\eta$, where $J=(a_p+a_pX+\frac{a_p}{p} X^2,X)\subset\Lambda$  for $\eta \neq 1$, \[\cha\X_\flat^\eta=\col_\flat^\eta(\mathbf{e})\cha\X_0^\eta,\]
\[\cha\X_\sharp^\eta=\begin{cases}\frac{\col_\sharp^\eta(\mathbf{e})}{X}\cha\X_0^\eta & \text{ if } \eta \neq 1\text{ and } a_p=0,\\\col_\sharp^\eta(\mathbf{e})\cha\X_0^\eta & \text{ if } \eta = 1\text{ or } a_p\neq 0.\end{cases}\]
\end{proof}

\begin{corollary}\label{almostmaintheorem}Given $\eta:\Delta\to \Z_p^\times$ and $n\in \N$, pick $*$ according to the Modesty Algorithm \ref{modestyalgorithm}. Then for large $n$, \[\nabla_n\X^\eta=\begin{cases}q_n^*+\Upsilon_n\X_*^\eta&\text{ if }a_p\neq0\text{ or }\eta= 1, \\ q_n^*+\Upsilon_n\X_*^\eta+1&\text{ if }a_p=0\text{ and }\eta\neq 1. \end{cases}\]
\end{corollary}
\begin{proof} This follows from Proposition \ref{kobayashisproposition}, Proposition \ref{above}, the Modesty Proposition \ref{modestypropn}, and \cite[Proposition 10.6 ii]{kobayashi}, which says that for large $n$, $\nabla_n\X_0^\eta=\Upsilon_n\X_0^\eta$.
\end{proof}
\end{subsection}

\begin{subsection}{The $\sharp/\flat$ - Selmer groups thus make the \v{S}afarevi\v{c}-Tate group grow: Main Theorem}
Note that in the Modesty Algorithm \ref{modestyalgorithm}, we always choose $*$ so that $\X_*^\eta$ is a finitely generated torsion $\Z_p[[X]]$-module. We have $\nabla_n \X_*^\eta=\mu_*^\eta(p^n-p^{n-1})+\lambda_*^\eta$ for large $n$ by \cite[Proposition 10.5.ii]{kobayashi}, where $\mu_*^\eta$ and $\lambda_*^\eta$ are the Iwasawa invariants of the characteristic polynomial of $\X_*^\eta$.

By Corollary \ref{almostmaintheorem} and equation (\ref{enandselmer}) in the proof of Remark \ref{upsilonisnabla}, we obtain our main theorem:

\begin{maintheorem}\label{maintheorem}Let $p$ be an odd prime of good supersingular reduction. Let $\eta:\Delta\rightarrow \Z_p^\times$ be a character. For $*\in\{\sharp,\flat\}$, denote by $\mu_*^\eta$ and $\lambda_*^\eta$ the $\mu-$ and $\lambda-$ invariants of $\X_*^\eta$ when $\X_*^\eta$ is finitely generated torsion as a $\Z_p[[X]]$-module, and set $\mu_*^\eta:=\infty$ if not. We denote by $r_\infty^\eta$ the rank of $E(K_\infty)^\eta$. %, finite by results of Kato and Rohrlich (\cite{kato} and \cite{rohrlich}).
Assume that $\Sha(E/K_n)[p^\infty]$ are finite for all $n$, and put $e_n^\eta=\ord_p(\Sha(E/K_n^\eta)).$ Then for large $n$, we have:
\begin{itemize}
\item
\text{ If $\mu_\sharp^\eta=\mu_\flat^\eta$ or $a_p=0$, then}
\[e_n^\eta-e_{n-1}^\eta=\begin{cases} (p^n-p^{n-1})\mu_\sharp^\eta+\lambda_\sharp^\eta-r_\infty^\eta+\begin{cases} q_n^\sharp &\text{ if $\eta=1$ or $a_p\neq 0$}\\ q_n^\sharp+1&\text{ if $\eta\neq1$ and $a_p=0$} \end{cases}\text{ if $n$ is odd},\\(p^n-p^{n-1})\mu_\flat^\eta+\lambda_\flat^\eta-r_\infty^\eta+q_n^\flat \text{ if $n$ is even.}\end{cases}\]
\item
\text{ If $\mu_\sharp^\eta<\mu_\flat^\eta $ and $a_p \neq 0$, then }
\[e_n^\eta-e_{n-1}^\eta=(p^n-p^{n-1})\mu_\sharp^\eta+\lambda_\sharp^\eta-r_\infty^\eta+q_n^\sharp.\]
\item
\text{ If $\mu_\sharp^\eta>\mu_\flat^\eta $ and $a_p \neq 0$, then}
\[e_n^\eta-e_{n-1}^\eta=(p^n-p^{n-1})\mu_\flat^\eta+\lambda_\flat^\eta-r_\infty^\eta+q_n^\flat.\]
\end{itemize}

\end{maintheorem}
\end{subsection}
\end{section}

\begin{section}{Proof of the Modesty Proposition}\label{proofofthemodestyproposition}
We now prove the Modesty Proposition \ref{modestypropn}. Subsection \ref{interestingcalculations}. contains the most important idea, that of valuation matrices.
Given a positive integer $n$, let $\Phi_n=\Phi_n(1+X)$ and explicitly write the elements of $\H_n$ from Definition \ref{Hn} as:
\[\links \omega_n^\sharp & \Phi_n\omega_{n-1}^\sharp \\ \omega_n^\flat & \Phi_n\omega_{n-1}^\flat \rechts:=\H_n\]
\begin{definition}Fix a basis $\mathbf{e}$ for the rank one $\Lambda$-module $\mathbf{H}^1(T)$, and denote its image under the localization map by $\mathbf{e}$ as well. Put $f_\sharp:=\col_\sharp(\mathbf{e}), f_\flat:=\col_\flat(\mathbf{e})$.
\end{definition}
\begin{definition}Choose a unit $u\in \Z_p^{\times}.$  We put \[M:=M_u:=\frac{\Lambda}{\omega_n^\sharp f_\sharp+u\omega_n^\flat f_\flat+\Phi_n\omega_{n-1}^\sharp f_\sharp+u\Phi_n\omega_{n-1}^\flat f_\flat},\text{ and } \\ M_{u,m}:=M\otimes \Lambda_m. \]
\end{definition}
\begin{subsection}{Valuation matrices}\label{interestingcalculations}
The goal of this subsection is to prove the following proposition, which computes the Kobayashi ranks of $M$ for large $n$. (They are independent of $u$. That is why we dropped the subscript $u$):
\begin{propn}\label{thestuff}Let $\eta:\Delta\rightarrow\Z_p^\times$ be a character. Then $\nabla_nM^\eta=q_n^*+\nabla_n(\Lambda^\eta/f_*^\eta)$ for large $n$, where $*$ is chosen according to the Modesty Algorithm \ref{modestyalgorithm}.

% $\text{ where $*$ is  }\begin{cases}\sharp \text{, if } \mu^\eta_\sharp=\mu^\eta_\flat \text{ and } n \text{ is odd} \text{ or } \mu^\eta_\sharp<\mu^\eta_\flat, \\
%\flat \text{, if }  \mu^\eta_\flat=\mu^\eta_\sharp \text{ and } n \text{ is even}\text{ or } \mu^\eta_\flat<\mu^\eta_\sharp.\\
%\end{cases} $
\end{propn}

To prove this proposition, we use the following idea.
\begin{definition}\label{valuationmatrix}Let $A=\links a & b \\ c & d \rechts$ be a matrix with entries in $\overline{\Q_p}$. We call \[\ord_p(A)=\bai \ord_p(a) & \ord_p(b) \\ \ord_p(c) & \ord_p(d) \dai\] the \textit{valuation matrix} of $A$. 

Note that we have an unnatural multiplication of valuation matrices: \[\ord_p(A)\ord_p(A')=\bai \min(\ord_p(aa'),\ord_p(bc')) & \min(\ord_p(ab'), \ord_p(bd'))\\ \min(\ord_p(a'c),\ord_p(c'd)) & \min(\ord_p(b'c),\ord_p(dd')) \dai\]
\end{definition}
\begin{lemma}Let $v:=\ord_p(a_p)$. Then the left column of $\ord_p(\H_n(\zeta_{p^n}-1))$ is
\[\bai v+p^{-2}+p^{-4}+\cdots +p^{2-n}\\ p^{-1}+p^{-3}+p^{-5}+\cdots +p^{1-n} \end{array}\right. \text{if $n$ is even}, \]

\[\bai p^{-1}+p^{-3}+p^{-5}+\cdots +p^{2-n}\\ v+p^{-2}+p^{-4}+\cdots +p^{1-n} \end{array}\right. \text{if $n$ is odd.}
 \]
\end{lemma}
\begin{proof} Multiply the valuation matrices of the factors of $\H_n(\zeta_{p^n}-1)$. Note that $\ord_p(\CCC_i(\zeta_{p^n}-1))=\bai v & 0 \\ p^{i-n} & \infty \dai$ for $i<n$ and $\ord_p(\CCC_n(\zeta_{p^n}-1))=\bai v & 0 \\ \infty & \infty \dai$. Denoting by $?$ a non-determined but unimportant value ,
\[\ord_p(\H_n(\zeta_{p^n}-1))=\bai v & 0 \\ p^{1-n} & \infty \dai\bai v & 0 \\ p^{2-n} & \infty \dai\cdots \bai v & 0 \\ p^{-1} & \infty \dai \bai 0 & ? \\ \infty &\infty \dai.\]
The lemma follows from the above unnatural multiplication.
\end{proof}

\begin{proof}[Proof of Proposition \ref{thestuff}]
Note that \[\nabla_n M^\eta =\ord_{\zeta_{p^n}-1}(\omega_n^\sharp f_\sharp^\eta+u\omega_n^\flat f_\flat^\eta)(\zeta_{p^n}-1)=(p^n-p^{n-1})\ord_p(\omega_n^\sharp f_\sharp^\eta+ \omega_n^\flat f_\flat^\eta)(\zeta_{p^n}-1).\] 
Now for $f\in \Z_p[[X]]$ with Iwasawa invariants $\mu$ and $\lambda$, we have $(p^n-p^{n-1})\ord_p(f(\zeta_{p^n}-1))=(p^n-p^{n-1})\mu+ \lambda$ for large $n$ by the $p$-adic Weierstra\ss\,  Preparation Theorem. The above lemma allows us to compute \[{(p^n-p^{n-1})\ord_p(f_\sharp^\eta(\zeta_{p^n}-1),f_\flat^\eta(\zeta_{p^n}-1))\hidari \omega_n^\sharp(\zeta_{p^n}-1)\\ \omega_n^\flat(\zeta_{p^n}-1)\migi}\] for large $n$, keeping in mind that $v=\infty$ if $a_p=0$, and $v=1$ if $a_p\neq 0 $ and $p$ is supersingular, because of the Weil bound $|a_p|<2\sqrt{p}$. Doing the calculations for even and odd $n$, we thus have \[(p^n-p^{n-1})\ord_p(\omega_n^\sharp f_\sharp^\eta+\omega_n^\flat f_\flat^\eta)(\zeta_{p^n}-1)=q_n^*+\nabla_n(\Lambda^\eta/f_*^\eta).\]\end{proof}
\end{subsection}

\begin{subsection}{Looking for a nice Iwasawa module}
In this section, we show that the Kobayashi rank of the module $\X_{loc}^\eta$ agrees with the Kobayashi rank of $M^\eta$ for large $n$, thereby finishing the proof.

For the case $a_p=0$, Kobayashi was able to decompose $E(\gm_n)$ into $E^\pm(\gm_n)$ (up to a constant overlap $E(\gm_{-1})$, cf. \cite[Proposition 8.22]{kobayashi}), which was at the heart of Kobayashi's analysis. This seems to be very hard in the case $a_p\neq0$ (cf. \cite[Open Problem 7.22]{shuron}). In what follows, we bypass this decomposition via Proposition \ref{auxiliary} and generalize Kobayashi's arguments centering around the exact sequence between \cite[page 33, (10.38)]{kobayashi} and \cite[page 33, (10.39)]{kobayashi} by considering $\Z_p[\Delta]$-module structures rather than just looking at isotypical components (i.e. $\Z_p$-module structures). 
\lemma We can describe the cokernel of $\col_0$ explicitly: $\coker \col_0\cong \Z_p[\Delta].$
\proof We have by \cite[Theorem 2.2 (2)]{shuron}\[\sum_{\tau\in\Delta}{\delta_0^1}^\tau=\Tr_{k_0/\Q_p}\delta_0^1=(a_p-2)\delta_0^0.\] Thus, we have \[P_{\delta_0^0}(z)=\sum_{\tau\in\Delta}\sum_{\sigma\in\Delta}\frac{({\delta_0^1}^{\tau\sigma},z)_0\sigma}{a_p-2}=\sum_{\tau\in\Delta}\sum_{\sigma\in\Delta}\frac{({\delta_0^1}^{\tau\tau^{-1}\sigma},z)_0\tau^{-1}\sigma}{a_p-2}=\frac{1}{a_p-2}\sum_{\tau\in\Delta}\tau^{-1} P_{\delta_0^1}(z).\]
We can thus write \[\col_0=(-a_p P_{\delta_0^1}+P_{\delta_0^0},-P_{\delta_0^1})=((-a_p+\frac{1}{a_p-2}\sum_{\tau\in\Delta}\tau^{-1}) P_{\delta_0^1},-P_{\delta_0^1}).\] We know that $-P_{\delta_0^1}$ is surjective (\cite[Proof of Proposition 7.3]{shuron}), so the map \[(x,y)\mapsto x+(\frac{1}{a_p-2}(\sum_{\tau^\in\Delta}\tau^{-1})-a_p)y\] induces an isomorphism between
$\frac{\Lambda_0\oplus \Lambda_0}{\im\col_0}$ and $\Lambda_0\cong \Z_p[\Delta]$.
But $\ker h_0=0$, since $h_0$ is multiplication by the matrix $\H_0=\links 0 & 1 \\ -1 & -a_p \rechts$. We thus have indeed
$\coker\col_0\cong L_0\cong(\Lambda_0\oplus \Lambda_0)/\im\col_0.$

\begin{proposition}\label{auxiliary}The following is a short exact sequence:
\[0 \rightarrow {E(k_n)\otimes \Z_p}\rightarrow H^1(k_n,T)\xrightarrow{\col_n}L_n\rightarrow \Z_p[\Delta]\rightarrow 0\]
\end{proposition}
\begin{proof}We first prove that $\ker \col_n=E(k_n)\otimes\Z_p$: We know $\ker \col_n=\ker P_{n,{\delta_n^1}}\cap \ker P_{n,{\delta_n^0}}$ from Theorem \ref{colemanmap}. Since $\delta_n^1$ and $\delta_n^0$ are generators for $\widehat{E}(\gm_n)$ as a $\Lambda_n$-module, we thus have for $z\in H^1(k_n,T)$: $z\in H^1(k_n,T)\iff (x,z)_n=0$ for all  $x \in \widehat{E}(k_n)\iff z \in E(k_n)\otimes \Z_p,$ cf. \cite[Example 1.8.4]{rubin}.

For the assertion about the cokernel, we reduce to the case $n=0$ via Nakayama's Lemma: 

Denote the projection $L_n\rightarrow L_0$ by $\pi$. By compatibility of the Coleman maps (cf. Theorem \ref{compatibility}), we have $\im \col_n\subset \pi^{-1}\im\col_0$. In the commutative diagram with exact sequences
\[\xymatrix {H^1(k_n,T) \ar@{->}[r]^{\col_n}\ar@{->>}[d]&\pi^{-1}\im\col_0\ar[d]\ar[r] &\frac{\pi^{-1}\im\col_0}{\im\col_n}\ar[r]\ar[d]&0\\
H^1(k_0,T) \ar[r]^{\col_0}&\im\col_0\ar[r]&0,}
\]
we have $\pi^{-1}\im\col_0=\im\col_n$ by surjectivity of the first vertical map (see proof of Proposition 7.3 in \cite{shuron}) and by Nakayama's Lemma. But $\pi$ induces an isomorphism \[\frac{L_n}{\pi^{-1}\im\col_0}\xrightarrow{\cong}\frac{L_0}{\im\col_0}\] by construction, so we have indeed reduced the problem to the case $n=0$.
\end{proof}

\begin{definition}Let $\X_{\col}=(\X_{\col,n})_n$ be the inverse system under the projection maps, where $\X_{\col,n}$ is the cokernel of the composition
\[\mathbf{H}^1(T)_{\Gamma_n}\rightarrow \frac{H^1(k_n,T)}{E(k_n)\otimes \Z_p}\xrightarrow{\col_n}L_n.\]
\end{definition}

\begin{corollary}We have an exact sequence of $\Lambda_n$-modules
\[0 \rightarrow \X_{loc}(E/K_n)\rightarrow \X_{\col,n}\rightarrow \Z_p[\Delta]\rightarrow 0.\]
\end{corollary}
\begin{proof}This follows from Proposition \ref{auxiliary}. 
\end{proof}

\begin{corollary}\label{checkthis}For any character $\eta:\Delta\rightarrow \Z_p^\times,$ we have $\nabla_n\Xloc^\eta=\nabla_n\X_{\col}^\eta-1.$
\end{corollary}
\begin{proof}This follows from the additivity of Kobayashi ranks. See \cite[Lemma 10.4.i]{kobayashi}.
\end{proof}

We now look at the following projective systems of $\Z_p$-modules:
\begin{itemize}
%\item $L:=(L_m)_m$, where $L_m=\frac{\Lambda_m\oplus \Lambda_m}{\ker h_m}$, and the projection maps,
\item $\Lambda=(\Lambda_m)_m$ and the projection maps, and
\item $\Lambda^\natural:=\left(\Lambda_m/(\omega_n^\sharp+\Phi_n\omega_{n-1}^\sharp, \omega_n^\flat+\Phi_n\omega_{n-1}^\flat)\right)_m$ and the projection maps.
\end{itemize}

\begin{propn}\label{randomlabel}Let $u\in\Z_p^\times$. Looking at the map $[+]_u\circ h_n:L_m\rightarrow \Lambda_m$ sending the class of $(a,b)$ to $(a(\omega_n^\sharp+\Phi_n\omega_{n-1}^\sharp)+ub(\omega_n^\flat+\Phi_n\omega_{n-1}^\flat))$ with kernel $\ker([+]_u\circ h_n)_m$, we obtain the following commutative diagram with exact sequences:
\end{propn}

%For $m\leq n$, we can then look at the sequence
%\[0 \rightarrow \ker([+]_u\circ h_n)_m\rightarrow L_m\xrightarrow{[+]_u\circ h_n} \Lambda_m\rightarrow \Lambda^\natural_m\rightarrow0,\]
%where the middle map is the map $[+]_u\circ h_n:L_m\rightarrow \Lambda_m$ sending the class of $(a,b)$ to $(a(\omega_n^\sharp+\Phi_n\omega_{n-1}^\sharp)+b(\omega_n^\flat+\Phi_n\omega_{n-1}^\flat))$, with kernel $\ker([+]_u\circ h_n)_m$:
\[\xymatrix {0\ar[r]&\ker([+]_u\circ h_n)_n \ar@{->}[r]\ar@{->}[d]^{\pi_1}& L_n\ar[d]\ar[r] &\Lambda_n\ar[r]\ar[d]&\Lambda^\natural_n\ar[r]\ar[d]^{\pi_4}&0\\
0\ar[r]&\ker([+]_u\circ h_n)_{n-1} \ar[r]& L_{n-1}\ar[r]&\Lambda_{n-1\ar[r]}&\Lambda^\natural_{n-1}\ar[r]&0}
\]
%\begin{propn}\label{randomlabel}This sequence is exact if $m=n$ or $m=n-1$.\end{propn}
 \begin{proof} This follows from the definitions.
\end{proof}

\begin{lemma}\label{bijections}$\pi_1$ and $\pi_4$ are $\Z_p$-module isomorphisms. 
\end{lemma}
\begin{proof}We prove that for any character $\eta:\Delta\rightarrow \Z_p^\times$, the $\eta$-isotypical components $\varepsilon_\eta \pi_1$  and $\varepsilon_\eta \pi_4$ of $\pi_1$ and $\pi_4$ are bijective. For $\varepsilon_\eta\pi_4$, surjectivity follows from the surjectivity of $\Lambda_n^\eta\rightarrow \Lambda_{n-1}^\eta$. As for injectivity, note that the kernel is any multiple of $\omega_{n-1}=\omega_{n-1}(X)$, while \[\omega_{n-1}^\flat(\omega_n^\sharp+\Phi_n\omega_{n-1}^\sharp)-\omega_{n-1}^\sharp(\omega_n^\flat+\Phi_n\omega_{n-1}^\flat)=\omega_n^\sharp\omega_{n-1}^\flat-\omega_{n-1}^\sharp\omega_n^\flat=\omega_{n-1},\] so the kernel is contained in the $\Lambda_n^\eta$-ideal $(\omega_n^\sharp+\Phi_n\omega_{n-1}^\sharp,\omega_n^\flat+\Phi_n\omega_{n-1}^\flat)$.

For $\varepsilon_\eta \pi_1$, surjectivity follows from the set-up. For injectivity, we analyze the kernel explicitly. If $(x,y)\in \ker\varepsilon_\eta\pi_1$, then we know \[\omega_{n-1}^\sharp \pi_1(x)+u\omega_{n-1}^\flat \pi_1(y)\in (\omega_{n-1})=0,\] so setting $\omega_n=\omega_n(X)$, we have \[\Phi_n\omega_{n-1}^\sharp x + u\Phi_n\omega_{n-1}^\flat y \in (\omega_n).\] But we also know that \[(a_p\omega_{n-1}^\sharp-\Phi_{n-1}\omega_{n-2}^\sharp + \Phi_n\omega_{n-1}^\sharp)x+u(a_p\omega_{n-1}^\flat-\Phi_{n-1}\omega_{n-2}^\flat+\Phi_n\omega_{n-1}^\flat)y\in (\omega_n),\] so it follows that \[(a_p\omega_{n-1}^\sharp-\Phi_{n-1}\omega_{n-2}^\sharp)x+u(a_p\omega_{n-1}^\flat-\Phi_{n-1}\omega_{n-2}^\flat)y\in(\omega_n).\] Thus, $(x,y)\in\ker \varepsilon_\eta h_n$.
\end{proof}

\begin{propn}Let $\eta:\Delta\rightarrow \Z_p^\times$ be a character. Then $\dim_{\Q_p}\Q_p \otimes {\ker([+]_u\circ h_n)}^\eta_{n-1}=\dim_{\Q_p}\Q_p \otimes \varepsilon_\eta\Lambda^\natural_{n-1}+1$. \end{propn}
\begin{proof}Note that $\Lambda_{n-1}^\eta$ has no $p$-torsion, and neither does $L_{n-1}^\eta$: Let $x,y\in \Lambda_{n-1}^\eta$ and suppose the image $(\overline{x},\overline{y})\in L_{n-1}^\eta$ has $(p^m\overline{x},p^m\overline{y})=0$ in $L_{n-1}^\eta$. We would then have $h_{n-1}(p^mx,p^my)=0$ in $\Lambda_{n-1}^\eta$, and since $h_{n-1}$ is multiplication by the matrix $\H_{n-1}$, we would have $h_{n-1}(x,y)=0$ in $\Lambda_{n-1}^\eta$, whence $(\overline{x},\overline{y})=0$ in $L_{n-1}^\eta$. 

Thus, it suffices to show that $\dim_{\mathbb{F}_p}\mathbb{F}_p\otimes L_{n-1}^\eta=\dim_{\mathbb{F}_p} {\mathbb{F}_p}\otimes \Lambda_{n-1}^\eta+1.$
But \[{\mathbb{F}_p}\otimes L_{n-1}^\eta\cong {\mathbb{F}_p}\otimes\left( \Lambda_{n-1}^\eta/(\omega_{n-1}^+)\oplus \Lambda_{n-1}^\eta/(\omega_{n-1}^-)\right),\text{ where }\]\[\omega_{n-1}^+=\prod_{2 \leq m\leq n-1,\text{m$m$ even}}\Phi_m(1+X), \text{ and }\omega_{n-1}^-=\prod_{m\leq n-1,\text{$m$ odd}}\Phi_m(1+X).\]
Now note that \[\omega_{n-1}^+\equiv \prod_{m \text{ even}, 1\leq m\leq n-1}X^{p^m-p^{m-1}}\mod p\text{ and }\omega_{n-1}^-\equiv \prod_{m \text{ odd}, 1\leq m\leq n-1}X^{p^m-p^{m-1}}\mod p,\] so $\dim_{\mathbb{F}_p}{\mathbb{F}_p}\otimes L_{n-1}^\eta=\sum_{1\leq m\leq n-1}=p^{n-1}-1,$ while $\dim_{\mathbb{F}_p}{\mathbb{F}_p}\otimes \Lambda_{n-1}^\eta=p^{n-1}.$
 \end{proof}

\begin{corollary}$\nabla_n{\Lambda_n^\natural}^\eta=\nabla_n\ker([+]_u\circ h_n)^\eta-1.$
\end{corollary}
\begin{proof} This is a consequence of the additivity of Kobayashi ranks \cite[Lemma 10.4 i]{kobayashi}. \end{proof}
\begin{lemma}\label{welldefinedness} There is a unit $u \in \Z_p^\times$ so that we have $\im\col_i \cap \ker([+]_u\circ h_n)_i=0$ for any $i$ so that $n\geq i\geq0$.
\end{lemma}
\begin{proof}Since $\im \col_i= \pi^{-1}_{i/0}(\im \col_0)$, where  $\pi_{i/0}$ is the projection from $L_i$ to $L_0$ (see proof of Proposition \ref{auxiliary}), this follows once we know it for $i=n$. But for $i=n$, note that $\ker \col_n=\ker (P_{\delta_n^1}, P_{\delta_n^0})$ by Theorem \ref{colemanmap}, so our assertion is equivalent to $\ker P_{\delta_n^1+u\delta_n^0}=\ker (P_{\delta_n^1}, P_{u\delta_n^0})$, i.e. we would like to prove that there is a unit $u$ so that $P_{\delta_n^1}(z)+ P_{u\delta_n^0}(z)=0$ implies $P_{\delta_n^1}(z)=0$ and $P_{u\delta_n^0}(z)=0$. We actually prove that $P_{\delta_n^1}(z)\subset (\Phi_j(1+X))$  for any $1\leq j\leq n$ and $P_{\delta_n^1}(z)\subset (X)$, which is enough, since $\omega_n(X)=X\prod_{1\leq j\leq n}\Phi_j(1+X)$. 

Choose $u$ so that $\delta_n^1+u\delta_n^0\neq v\delta_n^{j-n}$ for any $v\in \Z\cap \Z_p$. We can do this by letting $u\in \Z_p^\times\backslash\Z$, which has infinitely many elements. Thus, we write $\delta_n^1+u\delta_n^0= a\delta_n^{j-n+1}+b\delta_n^{j-n}$ with $a,b \in \Z_p\backslash\{0\}$. Now $\im bP_{\delta_n^{j-n}}=\im P_{b\delta_n^{j-n}}\subset \Phi_j(1+X)\Lambda_n$ by \cite[Lemma 3.6]{shuron}, and since $\Lambda_n$ has no $p$-torsion, $\im P_{\delta_n^{j-n+1}}\subset \Phi_j(1+X)\Lambda_n$ as well. By induction, $P_{\delta_n^1}(z)\subset \Phi_j(1+X)\Lambda_n$. 

To see that $P_{\delta_n^1}(z)\subset X\Lambda_n$, we have to prove that $P_{\delta_0^{n+1}}(z)=0$. Recall that we are assuming $P_{\delta_n^1}(z)+ P_{u\delta_n^0}(z)=0$, which implies $P_{\delta_0^{n+1}}(z)+ P_{u\delta_0^n}(z)=0$. Write $\delta_0^{n+1}+u\delta_0^n=a\delta_0^1+b\delta_0^0$ for $a,b\in\Z_p$. Since we have infinitely many units at our disposal, we may assume we have chosen $u$ so that $a\neq0\neq b$ and $\frac{a}{a_p-2}\neq -\frac{b}{p-1}$. But now we have \[P_{\delta_0^0}(z)=\sum_{\sigma\in\Delta=\G_0}({\delta_0^0}^\sigma,z)_0\sigma=\varepsilon_{\text{$\mathbb{1}$}}(p-1)(\delta_0^0,z)_0\] by definition, where $\mathbb{1}$ denotes the trivial character, and further 
\[\varepsilon_{\mathbb{1}}P_{\delta_0^1}(z)=(a_p-2)\varepsilon_{\mathbb{1}}(\delta_0^0,z)_0.\]
Since the non-trivial component of $bP_{\delta_0^0}(z)$ vanishes, so does that of $aP_{\delta_0^1}(z)$. Further, \[\varepsilon_{\mathbb{1}}(aP_{\delta_0^1}(z)+bP_{\delta_0^0}(z))=\varepsilon_{\mathbb{1}}cP_{\delta_0^0}(z)=0\] for some $c\in\Z_p\backslash\{0\}$ by the choice of $u$. Thus $P_{\delta_0^1}(z)=P_{\delta_0^0}(z)=0$, so that $P_{\delta_0^{n+1}}(z)=0$.
%From the setup, it suffices to prove this for $i=n$. That is, we have to prove $P^\eta_{\delta_n^1+u\delta_n^0}(z)=0$ implies $P_x(z)=0$ for any $x\in \hat{E}(\gm_n)$. This we accomplish as follows: Note that for $0 \leq i \leq n$, $P_{\delta_n^{i-n}}\subset \Phi_i(1+X)\Lambda_n$ (cf. \cite[Lemma 3.6]{shuron}). Thus if we choose $u$ so that $\delta_n^1+u\delta_n^0\neq v\delta_n^{i-n}$, where $v\in \Z$ is a unit in $\Z_p$, we know that $P_{p^*\delta_n^1}(z)=P_{p^*\delta_n^0}(z)=0$, where $*$ denotes some non-negative integer. But $\hat{E}(\gm_n)$ and $L_n$ have no $p$-torsion, so $P_{\delta_n^1}(z)=P_{\delta_n^0}(z)=0$. As for the choice of $u$, note that the cardinality of $\Z_p^\times-\Z$ is infinite, so there is a $u$ that makes the argument work for each of the (finitely many) $i$.
\end{proof}

\begin{remark}By the table in \cite[Section 2]{shuron}, $u=1+a_p$ would work, but we stayed away from that explicit choice because we didn't want to involve the periodicity of the $\delta_n^i$ with respect to $i$.\end{remark}

\begin{propn}Let $u\in\Z_p^\times$ satisfy the conclusions of the above  Lemma \ref{welldefinedness}. The exact sequence from Proposition \ref{randomlabel}, together with the projection maps, induces the following commutative diagram with exact sequences of $\Z_p[\Delta]$-modules.
%, where the first and last projection maps $\pi_1$ and $\pi_4$ are bijections.

\[\xymatrix {0\ar[r]&\ker([+]_u\circ h_n)_n \ar@{->}[r]\ar@{->}[d]^{\pi_1}& \X_{\col,n}\ar[d]\ar[r] &M_{n,u}\ar[r]\ar[d]&\Lambda^\natural_n\ar[r]\ar[d]^{\pi_4}&0\\
0\ar[r]&\ker([+]_u\circ h_n)_{n-1} \ar[r]& \X_{\col,n-1}\ar[r]&M_{n-1,u\ar[r]}&\Lambda^\natural_{n-1}\ar[r]&0}
\]
\end{propn}
\begin{proof}Commutativity follows from the set-up, and well-definedness from the above Lemma \ref{welldefinedness}. The bijectivities follow from calculations with determinants. \end{proof}

%\begin{lemma}\label{exceptionallemma}Let $a_p=-p, \eta$ be nontrivial, and $n\equiv 0\mod 6$. Then $\im\col_{n-1}^\eta = \ker([+]_u\circ h_n)_{n-1}^\eta$ and $\im\col_{n}^\eta = \ker([+]_u\circ h_n)_{n}^\eta$. If $a_p=-p$ and $\eta$ is trivial or $n \not\equiv 0 \mod 6$, then $\im\col_{n-1}^\eta \cap \ker([+]_u\circ h_n)_{n-1}^\eta=0$ and $\im\col_{n}^\eta \cap \ker([+]_u\circ h_n)_{n}^\eta=0$\end{lemma}
\begin{corollary}$\nabla_n \X_{\col}^\eta=\nabla_n M^\eta+1$.\end{corollary}
\begin{proof} This is a consequence of the additivity of Kobayashi ranks \cite[Lemma 10.4 i]{kobayashi}. \end{proof}
\begin{corollary} $\nabla_n \X_{loc}^\eta=\nabla_n M^\eta$.\end{corollary}
\begin{proof}Apply Corollary \ref{checkthis}.\end{proof}
The Modesty Proposition follows from this last corollary, Proposition \ref{thestuff}, and Remark \ref{upsilonisnabla}. 
%\begin{corollary}If $a_p=-p, \eta$ is nontrivial and $n \equiv 0 \mod 6$, then $\nabla_n \X_{loc}^\eta=\nabla_n M^\eta-1$. In all other cases, we have $\nabla_n \X_{loc}^\eta=\nabla_n M^\eta$.\end{corollary}
%\begin{proof} If $a_p\neq-p$, then blablabla thus $\nabla_n \X_{\col}^\eta=\nabla_n M^\eta+1$. The result follows taking into account $\nabla_n\Lambda_n^\natural=\nabla_n\ker([+]_u\circ h_n)-1,$ which can be seen as follows. If $a_p=-p$ and $\eta $ is trivial or $n \not\equiv 0 \mod 6$, then the $\eta$-isotypical component of the commutative diagram in proposition is a commutative diagram of $\Z_p$-modules with exact sequences, since blablablablabla, so the result follows as above. When $a_p=-p,\eta$ is nontrivial and $n\equiv 0 \mod 6$, then by lemma \ref{exceptionallemma}, the $\X_{\col,n}$ and $\X_{\col,n-1}$ inject into $M_n$ and $M_{n-1}$ respectively. MORE DETAILS.\end{proof}
\end{subsection}

\end{section}
\begin{section}{Relation to the work of Perrin-Riou}
Our Main Theorem \ref{maintheorem} generalizes and explains the works of Nasybullin and Perrin-Riou.
Nasybullin \cite{nasybullin} announced similar formulas almost four decades ago, some even in a more general context. He used unspecified pairs of integers and suggestively denoted them by $\mu$ and $\lambda$, but it seems he gave no proofs. A translation of this article in Russian is available upon request.

About a decade ago, Perrin-Riou \cite{perrinriou} found similar formulas for the $p$-primary part of $\Sha$ along the cyclotomic $\Z_p$-extension of $\Q$ by building on the work of Kurihara \cite{kurihara}. She constructed her pairs of $\mu$ and $\lambda$-invariants as limits of certain invariants of sequences of polynomials. 

\begin{theorem}[(Perrin-Riou \cite{perrinriou}, Th\'{e}or\`{e}me 6.1 (1),(4))] \footnote{We have corrected some typos, e.g. we have reinserted a necessary condition in brackets. Note that $K_n=\Q_n$ in her statement, following Kurihara's notation in \cite{kurihara}.}
Let $s$ be the rank of $E(\Q_\infty)$ and $p$ be a supersingular prime of good reduction and suppose $\Sha(E/\Q_n)$ is finite for all $n$. Then there are integers $\mu_+',\mu_-',\lambda_+',\lambda_-'$
 and a rational $\nu$ so that {\rm [if $a_p=0$ or we knew a priori that $\mu_+'=\mu_-'$]},
\[\ord_p\#\Sha(E/\Q_n)=\frac{p^{2[\frac{n}{2}]+1}}{p+1}\mu_+'+\frac{p^{2[\frac{n+1}{2}]+1}}{p+1}\mu_-'+\left[\frac{p}{p^2-1}p^n\right]+(\lambda_+'-s)\left[\frac{n}{2}\right]+(\lambda_-'-s)\left[\frac{n+1}{2}\right]+\nu.\]
\end{theorem}

Her integers come from the invariants in the first case of our Main Theorem \ref{maintheorem} with $\eta=\mathbb{1}$: \[\mu_+'=\mu_\sharp^\eta,\quad \mu_-'=\mu_\flat^\eta,\quad \lambda_+'=\lambda_\sharp^\eta,\quad \lambda_-'=\lambda_\flat^\eta-1. \footnote{The discrepancy between $\lambda_-'$ and $\lambda_\flat^\eta$ comes about because $\left[\frac{p}{p^2-1}p^n\right]-\left[\frac{p}{p^2-1}p^{n-1}\right]$ is equal to $q_n^\sharp+1=\left[\frac{p^n}{p+1}\right]+1$ for odd $n$, while it matches $q_n^\flat=\left[\frac{p^n}{p+1}\right]$ for even $n$.}\]
%\begin{conjecture} $\min(\mu_\sharp,\mu_\flat)=0$.\end{conjecture}
%Equivalence of conjectures. Numerical data, following Perrin-Riou and Pollack's tables.
One remarkable calculation \cite[Proposition 5.2]{perrinriou} essentially gives her formula when $a_p=0$ or $\mu_+'=\mu_-'$. This condition is necessary for her above theorem, but she mentions the existence of a similar formula when $a_p\neq0$ and $\mu_+'\neq \mu_-'$. In this case, the Iwasawa-theoretic meaning  is different: \textit{Both} pairs of invariants come from the characteristic ideal with the \textit{smaller} $\mu$-invariant.

We sketch for the interested reader how to treat this case using her methods. Assume $\mu_+'<\mu_-'$ with $a_p\neq0$. By the arguments of her proof of lemme 5.4 (cas 3, pg. 167), we have $\mu_-=\mu_++1$, and by lemme 5.4, only $\lambda_+$ plays a role when modifying the methods on pg. 168. The term $\left(\mu_{\epsilon(j)}+\frac{p}{p^2-1}+\frac{\lambda_{\epsilon(j)}}{p^{j-1}(p-1)}\right)$ becomes in our case \[\left(\mu_++\frac{p}{p^2-1}+\frac{\lambda_+}{p^{j-1}(p-1)}\right)\text{ or } \left(\mu_-+\frac{1}{p^2-1}+\frac{\lambda_+}{p^{j-1}(p-1)}\right)=\left(\mu_++\frac{p^2}{p^2-1}+\frac{\lambda_+}{p^{j-1}(p-1)}\right)\] depending on the parity of $j$, and thus depends only on $\mu_+$ and $\lambda_+$ in any case - but without a tandem Iwasawa theory at hand, it might seem as if the growth of the \v{S}afarevi\v{c}-Tate group were controlled by two pairs of invariants that \textit{happen to be the same}. 
\end{section}

\begin{acknowledgements}
We thank Barry Mazur, Robert Pollack, and Joseph Silverman for various improvements on an earlier draft of this paper, and Masato Kurihara, Shinichi Kobayashi, and our Topics Exam committee Hee Oh, Michael Rosen, and Joseph Silverman for encouragement, advice, and helpful discussions.
\end{acknowledgements}

\end{document}